\documentclass[12pt]{amsart}

\usepackage[utf8]{inputenc}
\usepackage{graphicx}
\usepackage{graphicx,color}
\usepackage{latexsym}
\usepackage[all]{xy}
\usepackage{mathrsfs}
\usepackage{enumerate}
\usepackage[shortlabels]{enumitem}
\usepackage{multicol}
\usepackage{lipsum}
\usepackage{amssymb}
\usepackage{tikz}
\usetikzlibrary{cd}
\usetikzlibrary{patterns}
\definecolor{DimGray}{rgb}{0.41, 0.41, 0.41}
\usepackage{float}
\usepackage{bm}
\usepackage{mathptmx}
\usetikzlibrary{shapes.geometric}
\usepackage{dsfont}
\usepackage{cjhebrew}
\usepackage[alpine]{ifsym}
\usepackage{xpicture}
\usepackage{calculator}
\usepackage{graphicx}
\usepackage{pgf,tikz,pgfplots}
\pgfplotsset{compat=1.15}
\usetikzlibrary{arrows}

\usepackage{amsmath}
\usepackage{amsfonts}
\usepackage{amssymb,enumerate}
\usepackage{amsthm}
\usepackage{fancyhdr}
\usepackage{hyperref}
\usepackage{cleveref}

\usepackage{caption}
\usepackage{subcaption}

\def\sideremark#1{\ifvmode\leavevmode\fi\vadjust{\vbox to0pt{\vss 
      \hbox to 0pt{\hskip\hsize\hskip1em           
 \vbox{\hsize2cm\tiny\raggedright\pretolerance10000
 \noindent #1\hfill}\hss}\vbox to8pt{\vfil}\vss}}} %
\theoremstyle{plain}
\newtheorem{theorem}{Theorem}[section]
\newtheorem{proposition}[theorem]{Proposition}

\newtheorem*{theorem*}{Theorem}

\newtheorem{corollary}[theorem]{Corollary}

\theoremstyle{definition}
\newtheorem{defin}[theorem]{Definition}

\newtheorem{ex}[theorem]{Example}

\theoremstyle{remark}
\newtheorem{rem}[theorem]{Remark}

\def\lcm{\operatorname{lcm}}
\newcommand{\Sg}{\mathbb{S}}

\counterwithin{figure}{section}
\numberwithin{equation}{section}

\title[Free semigroups and iterated torus knots]{Topological Representations of Free Numerical Semigroups via Iterated Torus Knots}

\author{Patricio Almir\'on}

\author{Adri\'an Olivares-Fern\'andez}



\thanks{The first named author is partially funded by MCIN/AEI/10.13039/501100011033 and by ``ERDF -- A way of making Europe", grant PID2022-138906NB-C22. During the course of the investigation the first named author has also been supported by: Spanish Ministerio de Ciencia, Innovaci\'{o}n y Universidades, grant RYC2021-034300-I by MI-CIU/AEI/10.13039/501100011033, by the programm IMAG–Maria de Maeztu grant CEX2020-001105-M / AEI /10.13039/501100011033, as well as Ministerio de Ciencia, Innovaci\'on y Universidades PID2020-114750GB-C32/AEI/10.13039/501100011033.}

\address{Departamento de \'Algebra, An\'alisis Matem\'atico, Geometr\'ia y Topolog\'ia; IMUVA (Instituto de Investigaci\'on en Matem\'aticas), Universidad de Valladolid\\ 
	Paseo de Bel\'en 7\\
	47011 Valladolid, Spain.}

\email{palmiron@uva.es}

\address{Programa de Doctorado en Ciencias, Escuela de Doctorado, Universitat Jaume I, Campus de Riu Sec,  12071.Caste\-ll\'on de la Plana, Spain}
\email{aolivare@uji.es}

\begin{document}
\begin{abstract}
In this paper we will associate a family \(\{K_1,\dots,K_l\}\subset \mathbb{S}^3\) of iterated torus knots to a given free numerical semigroup. We will describe the fundamental group of the knot complement of each knot of the family. 
Finally, we will show that all knots in the family have the same Alexander polynomial and it coincides (up to a factor) with the Poincaré series of the free numerical semigroup. As a consequence, we will provide families of iterated torus knots with the same Alexander polynomial of an irreducible plane curve singularity but which are non-isotopic to its associated knot.
\end{abstract}

\maketitle

\section{Introduction}

The problem of finding families of fibered knots that share the same polynomial invariant is a significant question in knot theory. One of the mainstream invariants is the Alexander polynomial. It was conjectured by Neuwirth and Burde \cite[Problems J and P, pg. 222]{Neu73} that the number of distinct fibered knots with a given Alexander polynomial is finite. However, this conjecture was shown to be false by Morton in 1978 \cite{Morton78}. Since then, Morton's work has inspired several studies providing families of infinitely many knots with the same Alexander polynomial and other invariants as well; see for example \cite{Bonahon,Kanenobu,kauffmanLopes17}. All methods rely in some way or another in the good properties of the Alexander polynomial under the satellization operation. 
\medskip

From this point of view, it is natural to ask to what extent one can assign a finite family of fibered knots to a certain combinatorial object from which the Alexander polynomial can be deduced. In addition, exploring how a purely combinatorial object, such as a free numerical semigroup, can be linked to a family of topological objects, like iterated torus knots, may reveal deeper connections between their inherent properties. Nowadays, the connection between free numerical semigroups and iterated torus knots is focused in some very concrete families \cite{CDGduke,CDG99b,CDG99a,Teragaito22,Wang18}. However, those results share the same idea, which can be summarized as: one semigroup for one knot. Our work, inspired by this idea, goes beyond it and is guided by the following question: for a fixed semigroup, how many knots can be assigned to it in such a way that they share the same Alexander polynomial?
\medskip

 A numerical semigroup \(S\subset \mathbb{N}\) is a finitely generated additive submonoid of the natural numbers with finite complement. If \(S\) is minimally generated by the set \(\{a_0,\dots,a_g\}\) then we will say that \(S\) is free if 
 \begin{equation*}
    \lcm(\gcd(a_0,\dots,a_{i-1}), a_i)\in \langle a_0, \dots, a_{i-1} \rangle \ \text{for all }   i\in\{1,\dots, g \}
\end{equation*}
for some ordering of the minimal set of generators. In particular, it may happen that there exist different possible orderings for which \(S\) is free. In this paper, our main goal is to construct a family of iterated torus knots \(\{K_1, \dots, K_l\} \subset \mathbb{S}^3\) whose cable invariants are computed from the minimal set of generators of \(S\), i.e. the cable invariants are \(\{e_{i-1}/e_i,a_i/e_i\}_{i=1,\dots,g}\) where \(e_i=\gcd(a_0,\dots,a_{i})\) for a fixed ordering. Obviously, this construction depends on the ordering of the minimal generators. Therefore, for a fixed semigroup \(S,\) it may produce different isotopy classes of iterated torus knots, see Example \ref{ex:firstexample}. Our first main result, Theorem \ref{teo:iterated_knot_fundamental_group} provides the construction of each iterated torus knot \(K_i\) and a presentation of its fundamental group:
\begin{theorem*}\ref{teo:iterated_knot_fundamental_group}
Let \(S=\langle a_{r_0},\dots, a_{r_g}\rangle\) be a free semigroup and let \(\mathcal{K}_S\) be the family of knots associated to \(S\). Let \(\mathcal{G}=\{a_0,\dots,a_{g}\}\in\mathcal{A}\) be an arrangement for which \(S\) is free and let \(K_{\mathcal{G}}\) be its associated knot.
Then, the fundamental group of the knot \(K_{\mathcal{G}}\) admits a presentation with $g+1$ generators and $g$ relations.
    \end{theorem*}
As a consequence, in Definition \ref{def:familyknots} we will provide the precise definition of the family of knots associated to \(S.\) 
\medskip
 
One of the main characteristics of the family of iterated torus knots associated to the free numerical semigroup \(S\) is that they all share the same Alexander polynomial, see Section \ref{sec:alexander}. Our second main result is that the Alexander polynomial of any member of the family coincides, up to a factor, with the Poincaré series of the free numerical semigroup, which in this case is nothing but the generating series of the semigroup \(S:\)
\begin{theorem*}\ref{thm:AlexanderPoincare}
    Let \(S=\langle a_0,\dots,a_g\rangle\) be a free numerical semigroup and let \(\mathcal{K}_S\) be the family of associated knots. Then for all \(K\in\mathcal{K}_S\) we have
    \[\Delta_{K}(t)=P_S(t)(1-t)=(1-t)\left(\sum_{s\in S}t^s\right).\]
    In particular, all knots associated to a free numerical semigroup have the same Alexander polynomial.
\end{theorem*}

The first prototypical example of a free numerical semigroup appears in the case of irreducible plane curve singularities. If \(\left(C,0\right)\subset \left(\mathbb{C}^2,0\right)\) is a germ of irreducible plane curve singularity then one can canonically associate a knot \(K_C=C\cap \mathbb{S}^3\) and a semigroup \(S_C,\) which is the set of intersection multiplicities with the curve; see \cite{zariski}. In \cite{CDG99a}, Campillo, Delgado and Gusein-Zade discovered the coincidence \(\Delta_{K_C}(t)=(1-t)P_{S_C}(t)\) which in fact implies the coincidence between the Poincaré series and the monodromy zeta function associated to \(C.\) Later on, they continued with several investigations on which Alexander polynomials of links associated to singularities can be regarded as Poincaré series of their ring of functions, see for example \cite{CDGduke,CDG04,CDG15,CDGMathZ}. The knots arising from irreducible plane curve singularities are called algebraic knots. Also, the semigroup \(S_C\) is a particular example of free semigroup \cite{zariski} (see also \ref{sec:algebraicknots}). Therefore, it makes sense to study in detail the family of knots associated to the semigroup \(S_C\)  of an irreducible plane curve. As we show in Section \ref{sec:algebraicknots}, our family brings richness to these connections and introduces a new perspective to the issue. We observe that, in general, our family contains knots which are non-isotopic to the knot of the curve \(K_C\) but share some of their invariants, such as the Alexander polynomial. Therefore, Theorem \ref{teo:iterated_knot_fundamental_group} provides a way to assign other types of knots to an irreducible plane curve in ``a non-canonical way".
\medskip

 Following Campillo, Delgado and Gusein-Zade's result, Wang \cite{Wang18} proposed the study of \(L\)--space knots, a knot that admits an \(L\)--space surgery. Ozsváth and Szabó \cite[Theorem 1.2]{OStopology} showed that for an \(L\)--space knot $\Delta_K(t)=\sum_{i=0}^{2n}(-1)^it^{\alpha_i}.$ From this, Wang \cite{Wang18} defined a formal semigroup, i.e a subset of $\mathbb{Z}_{\geq 0}$ not necessarily closed by addition, $S_K$ using the exponents of the Alexander polynomial that by construction implies the coincidence between \(\Delta_K(t)\) and the generating series of $S_K$. One of the problems with this approach, is that in general \(S_K\) could not be a monoid, i.e. closed by addition of natural numbers, but just a formal semigroup as showed by Wang in \cite[Example 2.3]{Wang18}. However, Wang shows that for any \(L\)--space knot which is an iterated torus knot then its formal semigroup \(S_K\) is a semigroup. As showed in \cite{Wang18}, algebraic knots are \(L\)--space iterated torus knots, but the converse is not true. On the other hand, Teragaito \cite{Teragaito22} showed that there exist hyperbolic \(L\)-space knots for which its formal semigroup is a semigroup. In section \ref{sec:Lspace}, we show that in the family of knots associated to a free semigroup there could exist \(L\)--space knots and non \(L\)--space knots. Moreover, thanks to a result of Hom \cite{Hom11} we can characterize those knots in the family which are \(L\)--space knots (Proposition \ref{prop:charL-space}). Hence, the general situation can be summarized as
 \[\left\{\text{algebraic knots}\right\}\subsetneq \left\{\text{\(L\)--space iterated torus knots}\right\}\subsetneq\left\{\begin{array}{c}
      \text{Family of knots associated}  \\
      \text{to a free numerical semigroup}
 \end{array} \right\}.\]

In conclusion, if \(S\) is a free numerical semigroup, our work provides a systematic study of how many knots can be assigned to it such that they share the same Alexander polynomial. Additionally, we recover well-known cases, such as algebraic knots and \(L\)--space knots, and extend their properties in a systematic way. Thus, we demonstrate the power of applying the combinatorics of a numerical semigroup to other fields, such as knot theory.

\section{The knots associated to a free semigroup}\label{sec:topologyknotfree}

Recall that a numerical semigroup \(S\subset\mathbb{N}\) is an additive submonoid of the natural numbers with finite complement, i.e. \(\mathbb{N}\setminus S\) contains only finitely many elements. Any numerical semigroup is finitely generated, i.e. there exists \(a_0,\dots, a_g\in \mathbb{N}\) such that 
\[S=\langle a_0,\dots, a_g\rangle=a_0\mathbb{N}+\cdots+a_g\mathbb{N}.\]

Following Bertin and Carbone \cite{beca77}, we will say that $S$ is a free semigroup for the arrangement $\{a_0,\dots,a_g\}$ if it satisfies the condition:
\begin{equation}\label{eq:freesemigroup}
	n_ia_i\in \langle a_0, \dots, a_{i-1}\rangle \ \text{for all }   i\in\{1,\dots, g \}
\end{equation}
where $e_0=a_0, \hspace{1mm} e_i=\text{gcd}(a_0,\dots, a_i)$ and $n_i=\frac{e_{i-1}}{e_i}$. A useful observation about this class is that can be constructed in an iterative way. Let us denote by \(S_i=\langle a_0,\dots,a_i\rangle/e_i\) the ``truncated numerical semigroups" of the arrangement. Observe that we have  \(S_2=n_2S_1+a_2/e_2\mathbb{N},\) \(S_3=n_3S_2+a_3/e_3\mathbb{N}\) and thus we can write
\[S=(n_2\cdots n_g)S_1+(n_3\cdots n_g)S_2+\cdots +a_g\mathbb{N}.\]
Therefore, one can see that the pairs of coprime integers \((n_i,a_i/e_i)\) are crucial for this construction. The main goal of this section will be to show that any free semigroup is topologically realizable; see \cite{Acisurvey},\cite{AMknots}. Moreover, we will show that the associated knot is an iterated torus knot with cabling invariants \(\{(n_i,a_i/e_i)\}_{i=1,\dots,g}.\)

\subsection{Iterated torus knots and splice diagrams}

In their 1985 monograph \cite{EN}, Eisenbud and Neumann provided the foundations of a powerful tool used to construct new links from old. This operation, called splice, encodes the particular cases of connected sum, cabling and disjoint sum among others. In this section, we will briefly recall the main properties of this operation and its usefulness in order to work with iterated torus knots. For further details about it we refer to \cite{EN}. For generalities about knots we refer to \cite{BurdeKnots,EN,Murasugibook} and the references therein.


Let $L\subset \Sg^3,L'\subset \Sg^3$ be two links and choose components $S\subset L$ and $S'\subset L'$. Take tubular neighborhoods of the components, namely $V(S)$ and $V(S')$  and choose $M,L\subset \partial V(S)$ and $M',L'\subset \partial V(S')$ topologically standard meridians and longitudes. Form $\Sg^3= (\Sg^3 \smallsetminus \text{int }V(S))\bigcup (\Sg^3 \smallsetminus \text{int }V(S'))$ by pasting boundaries matching $M$ with $L'$ and $L$ with $M'$. The link $(L\smallsetminus S)\cup (L'\smallsetminus S')$ is called the splicing of the links $L$ and $L'$ along $S$ and $S'$, it's denoted by 
$$\xymatrix{ L\ar@{-}[rr]_{S\hspace{1.4cm} S'}& &L'}.$$

Also, the links and the splice operation can be expressed using certain diagrams called ``splice diagrams''. They consist of three main ingredients: 
\begin{enumerate}
	\item[$\bullet$] Nodes of the form \(\begin{tikzpicture}[baseline=-0.5ex]
		\draw (0,0) circle [radius=0.3cm];
		\node at (0,0) {\(\epsilon\)};
	\end{tikzpicture}\) with the symbol \(\epsilon=+,-\). This node represents a Seifert manifold embedded in the link exterior with the orientation given by the corresponding symbol. 
	\item[$\bullet$] Boundary vertices, which are of the form \(\begin{tikzpicture}[baseline=-0.5ex]
		\draw (0,0) -- (1.5,0) node[pos=1, fill=black, circle, minimum size=5pt, inner sep=0pt]{};
	\end{tikzpicture}.\) They correspond to a solid torus representing the tubular neighborhood of some fiber in a Seifert manifold.
	\item[$\bullet$] Arrowhead vertices, which are of the form \(\rightarrow.\) They correspond to link components. 
\end{enumerate}
The edges incident to a node \(\begin{tikzpicture}[baseline=-0.5ex]
	\draw (0,0) circle [radius=0.3cm];
	\node at (0,0) {\(\epsilon\)};
\end{tikzpicture}\) correspond either to boundary components of the Seifert manifold, or to boundaries of the tubular neighborhoods of fibers. We refer to \cite[Chapter II]{EN} for further details and properties. 
\medskip

Along this paper we will be interested in the study of iterated torus knots. Therefore, in order to simplify the exposition we will only recall the splice diagram of that type of knot. Consider the map $f:S^1\to S^1 \times S^1$ given by $f(z)=(z^p, z^q)$; if we identify $\mathbb{S}^1$ with the complex unit circle, the image of $f$ is a closed curve, say $K$, which turns out to be a knot in $\mathbb{S}^3$ obtained by making $q$ winds longitudinally and $p$ winds transversally: this is a torus knot of type $(p,q)$. Obviously, it has only one link component and one Seifert manifold embedded in the link exterior, that can be positively oriented if we consider \(p,q\geq 0\). Figure \ref{fig:diagram_torus_knot} shows the splice diagram of a $(p,q)$-torus knot.
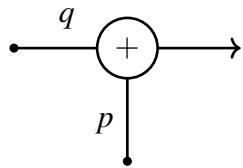
\begin{figure}[H]
	\centering  
	\begin{tikzpicture}
		\draw[line width=1pt] (0,0) circle [radius=0.4cm];
		
		\draw[line width=1pt] (-1.5,0) -- (-0.4,0);
		\draw[line width=1pt, ->] (0.4,0) -- (1.5,0);
		\draw[line width=1pt] (0,-1.5) -- (0,-0.4);
		
		\filldraw (-1.5,0) circle [radius=1.5pt];
		\filldraw (0,-1.5) circle [radius=1.5pt];
		
		\node at (0,0) {$+$};
		\node at (-0.8,0.4) {$q$};
		\node at (-0.3,-1) {$p$};
		
	\end{tikzpicture}
	
	\caption{Splice diagram of  a $(p,q)-$torus knot.}
	\label{fig:diagram_torus_knot}
\end{figure}

Let us now move to the splice operation. Take a knot $K\subset \Sg^3$ and a tubular neighborhood $V(K)$, for $p,q$ coprime integers consider the curve $K(p,q)$; which is the image in $\partial V(K)$ of the function $f$ defined in the previous paragraph. Replacing the knot $K$ by the knot obtained by $K(p,q)$ is called a cabling operation over $K$. We will also say that we have done a $(p,q)$--cable over $K.$ So we will say that a knot is an iterated torus knot if it can be obtained by iterating the cabling process over an unknot. For example, consider the pairs of coprime integers $\{(p_1,q_1),\dots, (p_n,q_n)\}.$ The knot obtained by doing a $(p_n,q_n)$ cable over a $(p_{n-1},q_{n-1})$ cable over a \dots  $(p_2,q_2)$ cable over a $(p_1,q_1)$ cable over an unknot is an iterated torus knot. Also, we will say that the previous knot has cabling invariants $\{ (p_i,q_i)\}_{i=1,\dots n}$. 
The splice diagram of an iterated torus knot with cabling invariants $\{ (p_i,q_i)\}_{i=1,\dots n}$ is the one of Figure \ref{fig:diagram_general_cabling}
\begin{figure}[H]
	\centering
	\begin{tikzpicture}
		
		\draw[line width=1pt] (0,0) circle [radius=0.3cm];
		\draw[line width=1pt] (2.5,0) circle [radius=0.3cm];
		\draw[line width=1pt] (5.5,0) circle [radius=0.3cm];
		
		\draw[line width=1pt] (-1.5,0) -- (-0.3,0);
		\draw[line width=1pt] (0.3,0) -- (2.2,0);
		\draw[line width=1pt] (0,-1.5) -- (0,-0.3);
		\draw[line width=1pt] (2.8,0) -- (3.2,0);
		\draw[line width=1pt,dashed] (3.2,0) -- (4.5,0);
		\draw[line width=1pt] (4.5,0) -- (5.2,0);
		\draw[line width=1pt,->] (5.8,0) -- (7,0);
		\draw[line width=1pt] (2.5,-1.5) -- (2.5,-0.3);
		\draw[line width=1pt] (5.5,-1.5) -- (5.5,-0.3);
		
		\filldraw (-1.5,0) circle [radius=1.5pt];
		\filldraw (0,-1.5) circle [radius=1.5pt];
		\filldraw (2.5,-1.5) circle [radius=1.5pt];
		\filldraw (5.5,-1.5) circle [radius=1.5pt];
		
		\node at (0,0) {$+$};
		\node at (2.5,0) {$+$};
		\node at (5.5,0) {$+$};
		\node at (-0.7,0.3) {$q_1$};
		\node at (-0.3,-0.7) {$p_1$};
		\node at (0.7,0.3) {$1$};
		\node at (1.9,0.3) {$q_2$};
		\node at (2.1,-0.7) {$p_2$};
		\node at (3.2,0.3) {$1$};   
		\node at (4.9,0.3) {$q_n$};
		\node at (5.1,-0.7) {$p_n$};
		\node at (6.2,0.3) {$1$};    
	\end{tikzpicture}
	\caption{Splice diagram of an iterated torus knot.}
	\label{fig:diagram_general_cabling}
\end{figure}
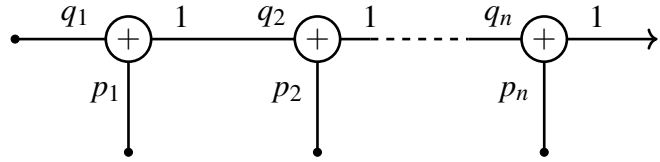

\subsection{The fundamental group of a knot associated to a free semigroup}

Consider a knot $K\subset \Sg^3$, we will denote by $V(K)$ a tubular neighborhood of $K$. The complement of $K$ is defined as the closure of $\Sg^3$ minus a tubular neighborhood of $K$, that is  $\overline{\Sg^3\smallsetminus V(K)}$, which we will denote by $\Sg^3_K$. Moreover, when we talk about the knot group of $K$ we mean the fundamental group of the complement of $K$, namely $\pi_1(\Sg^3_K)$. In order to proceed with our main result, we need first to fix some notation and introduce the notion of a knot associated to a free semigroup. 
\medskip

Let \(S\) be a free numerical semigroup minimally generated by \(\mathcal{G}=\{a_0,\dots,a_g\}.\) Let \(\mathcal{A}\) be the set of all possible arrangements of \(\mathcal{G}\) for which \(S\) is free. For a fixed arrangement, say for simplicity \(\mathcal{G}_i=\mathcal{G}=\{a_{0},\dots,a_{g}\}\in\mathcal{A},\) we already mentioned that the set of coprime integers $\{(n_i,\frac{a_{i}}{e_{i}})\}_{i=1,\dots, g }$ determines a recursive way to construct \(S.\) From this ordered set we can produce an iterated torus knot \(K_{\mathcal{G}}\) associated to the arrangement \(\mathcal{G}\) in the following iterative way. 
\medskip

For the first iteration we consider $K_1$ a $(n_1,\frac{a_1}{e_1})$-torus knot. Let us denote by $\widetilde{K_i}$ the $(n_i,\frac{a_i}{e_i})$-torus knot and by \(\widetilde{L}_i=\mathbb{S}^1\cup \widetilde{K}_i\) the link formed by the unknot and \(\widetilde{K}_i\) with linking number equal to \(n_i;\) its splice diagram is the one in Figure \ref{fig:diagram_auxlink}
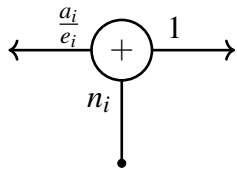
\begin{figure}[H]
	\centering  
	\begin{tikzpicture}
		\draw[line width=1pt] (0,0) circle [radius=0.4cm];
		
		\draw[line width=1pt, ->] (-0.4,0) -- (-1.5,0);
		\draw[line width=1pt, ->] (0.4,0) -- (1.5,0);
		\draw[line width=1pt] (0,-1.5) -- (0,-0.4);
		
		\filldraw (0,-1.5) circle [radius=1.5pt];
		
		\node at (0,0) {$+$};
		\node at (-0.3,-0.7) {$n_i$};
		\node at (0.7,0.3) {$1$};
		\node at (-0.7,0.3) {$\frac{a_i}{e_i}$};
	\end{tikzpicture}
	
	\caption{Splice diagram of  the links \(\widetilde{L}_i\)}
	\label{fig:diagram_auxlink}
\end{figure}

for \(i=2,\dots g.\) The second iteration is a splice of the knot $ K_1$ and the link $\widetilde{L_2}=\mathbb{S}^1\cup \widetilde{K_2} $ choosing the components $K_1$ and $\mathbb{S}^1\subset \widetilde{L_2}$, that is the splicing operation given by the diagram 
$$\xymatrix{  K_1\ar@{-}[rr]_-{K_1\hspace{1.4cm} \mathbb{S}^1}& &\widetilde{L}_2}.$$


The resulting knot, $ K_2$ is then a $\widetilde{K_2}$ cabling on $K_1$. Recursively, the \(i\)--th iteration is made by splicing $ K_{i-1}$, the knot obtained at the \(\left(i-1\right)\)--th iteration, and the link $ \widetilde{L_i}=\mathbb{S}^1\cup \widetilde{K_i} $ choosing the components $K_{i-1}$ and $\mathbb{S}^1\subset \widetilde{L_i}$, namely the splice operation is 
$$\xymatrix{ K_{i-1}\ar@{-}[rr]_-{K_{i-1}\hspace{1.2cm} \mathbb{S}^1}& &\widetilde{L_i}}.$$

After $g$ iterations we obtain an itereated torus knot \(K=K_{\mathcal{G}}\) associated to the free semigroup \(S\) and the arrangement \(\mathcal{G}\in \mathcal{A}\) of its minimal set of generators. Moreover, the knots \(K_i\) appearing in the process correspond to the knots of the truncated semigroups \(S_i.\) With this construction, it makes sense to define the following family of knots associated to a free semigroup.

\begin{defin}\label{def:familyknots}
	We define the family of knots associated to \(S\) as 
	\[\mathcal{K}_S:=\{K_{\mathcal{G}_i}\;|\; \mathcal{G}_i\in\mathcal{A} \}\]
\end{defin}

Our next goal is to compute the fundamental group of each knot \(K_{\mathcal{G}_i}\in\mathcal{K}_S\) contained in the family. To do that, we first introduce the following notation. We define the integers \(b_1=\frac{a_1}{e_1}\) and \(b_i=\frac{a_i}{e_i}-n_in_{i-1}\frac{a_{i-1}}{e_{i-1}}\) for \( i=2,\dots, g\). We remark that because $\gcd(n_i,\frac{a_i}{e_i})=1$ then $\gcd(b_i,n_i)=1$, so using the Bezout identity we can find $x_i,y_i\in \mathbb{Z},\hspace{1mm} x_i>0\hspace{1mm}$ satisfying $x_ib_i=y_in_i+1$. Under this notation we have the following

\begin{theorem}\label{teo:iterated_knot_fundamental_group}
	Let \(S=\langle a_{r_0},\dots, a_{r_g}\rangle\) be a free semigroup and let \(\mathcal{K}_S\) be the family of knots associated to \(S\). Let \(\mathcal{G}=\{a_0,\dots,a_{g}\}\in\mathcal{A}\) be an arrangement for which \(S\) is free and let \(K_{\mathcal{G}}\) be its associated knot, i.e. an iterated torus knot having cabling invariants $\{(n_i,\frac{a_i}{e_i})\}_{i=1,\dots, g }.$ Then, the fundamental group of the knot \(K_{\mathcal{G}}\) admits a presentation with $g+1$ generators, $P_1,Q_1,\cdots,Q_{g}$, and $g$ relations:
	\begin{equation}\label{eq:relations_groupknot}
		Q_i^{n_i}=P_i^{b_i}Q_{i-1}^{n_{i-1}n_i} \hspace{3cm}(i=1,2,\dots,g\hspace{1mm};\hspace{1mm}Q_0=1)
	\end{equation}
	and the elements $P_2,\dots,P_{g}$ are defined by the relations:
	\begin{equation}\label{eq:relations_Ploops}
		P_{i+1}P_i^{y_i}Q_{i-1}^{n_{i-1}x_i}=Q_i^{x_i}, \hspace{3cm}(i=1,2,\dots,g-1) 
	\end{equation}
	where $x_i,y_i\in\mathbb{Z},\hspace{1mm} x_i>0,$ defined applying the Bezout identity to the pairs $(n_i,b_i)$, i. e., satisfying $x_ib_i=y_in_i+1.$
	
\end{theorem}

\begin{proof}
	
	First of all, we consider the $(n_1,a_1/e_1)$-torus knot $K_1\subset \Sg^3$ and $V(K_1)$ a tubular neighborhood of $K_1$, then it is known (see for example \cite[Prop. 3.28]{BurdeKnots}) that the fundamental group can be expressed as follows
	\begin{equation*}\label{eq:group_torus_knot}
		\pi_1(\Sg^3_K)= \langle P,Q\hspace{1mm}:\hspace{1mm} P^{a_1/e_1}=Q^{n_1}\rangle
	\end{equation*}
	
	where $P$ is homotopic to a meridian of the torus on which $K$ is defined  and $Q$ is the center line (core) of that torus.
	\medskip
	
	Secondly,  let us elaborate on the case $g=2$, and then by induction we will proceed with the general case. Let us consider the knot $K_2$ obtained after the second iteration and $\hat{V}$ a tubular neighborhood of $K_2$. The next step is to separate $\overline{\Sg^3\smallsetminus \hat{V}}$ into two spaces $U=\overline{\Sg^3\smallsetminus V}$ and $W=\overline{V\smallsetminus \hat{V}}$, here $V$ is a tubular neighborhood of the knot $K_1$ chosen so that $\hat{V}$ is contained inside. It is obvious that $U$ is the closure $\Sg^3$ minus a tubular neighborhood of the knot $K_1$, $W$ is homeomorphic to the closure of an unknoted solid torus $T$ minus a tubular neigborhood of $\widetilde{K_2}$ and that $U\cap W$ is homeomorphic to a torus. Therefore, applying the Seifert-Van Kampen theorem we obtain
	\begin{equation}
		\pi_1(\overline{\Sg^3\smallsetminus \hat{V}})=\pi_1(\overline{\Sg^3\smallsetminus V})\ast_{\langle\lambda,\mu\rangle}\pi_1(\overline{T\setminus \hat{V}})
	\end{equation}
	
	where $\lambda,\mu$ represent the homotopy class of a meridian and longitude of $T.$\\
	
	It is straightforward to check that the fundamental groups of each part admit the following representations:
	\begin{equation*}
		\pi_1(\overline{\Sg^3\smallsetminus V})=\langle  P_1, \ Q_1\ :\ Q_1^{n_1}=P_1^{\frac{a_1}{e_1}}\rangle    \quad \text{and}\quad 
		\pi_1(\overline{T\setminus \hat{V}})=\langle P_2,\  Q_2, \ L\ :\ Q_2^{n_2}P_2^{-\frac{a_2}{e_2}}=L^{n_2},\hspace{1mm} \left[L,P_2 \right]=1\rangle,
	\end{equation*}
	where $\left[\hspace{1mm}, \hspace{1mm}\right]$ denotes the commutator.
	
	\medskip
	
	At this point, we just have to check the classes of the meridian \(P_2\) and longitude \(L\) of \(T\). First, we observe that $P_2$ is also a meridian of $V$ and thus a straightforward but tedious computation leads to the relation 
	$P_2=Q_1^{x_1}P_1^{-y_1}$ in $\pi_1(U)$,  where $x_1,y_1\in \mathbb{Z}, \hspace{1.5mm} x_1>0$ are such that $x_1b_1=y_1n_1+1$. For a guide to the computation see for example the particular case in \cite{Zar32}. Recall also that \(b_1=a_1/e_1\) and \(b_i=a_i/e_i-n_in_{i-1}a_{i-1}/e_{i-1}\) for \(i=2,\dots,g.\)
	
	\medskip
	
	For the longitude \(L\), we have that \(L\) corresponds to a topological standard longitude in $\partial V$. Therefore we obtain the following relation
	\[L=Q_1^{n_1}P_2^{-n_1b_1},\]
	from this relation and, since $Q_1^{n_1}$ is a central element in $\pi_1(\overline{\Sg^3\smallsetminus V})$, the fact that $Q_1^{n_1}$ commutes with $P_2$ we obtain 
	\begin{equation}\label{eq:firstrelation}
		Q_2^{n_2}=L^{n_2}P_2^{a_2/e_2}=Q_1^{n_2n_1}P_2^{\frac{a_2}{e_2}-n_1n_2\frac{a_1}{e_1}}=Q_1^{n_2n_1}P_2^{b_2}.
	\end{equation}

	Once the first step is understood, let us move to the induction hypothesis. Assume that the fundamental group of the iterated torus knot $K_{g-1}$ having cabling invariants $\{(n_i,\frac{a_i}{e_i})\}_{i=1,\dots, g-1}$ admits a presentation with $g$ generators, $P_1,Q_1,\cdots,Q_{g-1}$, and $g-1$ relations:
	\begin{equation}
		Q_i^{n_i}=P_i^{b_i}Q_{i-1}^{n_{i-1}n_i} \quad \text{for}\quad i=1,2,\dots,g-1,
	\end{equation}
	where \(Q_0=1\) and the elements $P_2,\dots,P_{g-1}$ are defined by the relations: \( P_{i+1}P_i^{y_i}Q_{i-1}^{n_{i-1}x_i}=Q_i^{x_i} \) for \(i=1,2,\dots,g-2,\) with $x_i,y_i\in\mathbb{Z}, x_i>0$ satisfying $x_ib_i=y_in_i+1.$
	\medskip
	
	Again we will use the Seifert-Van Kampen theorem. In this case, we take the knot $K_g$ and $V(K_{g})$ a tubular neighborhood and we separate $\Sg^3_{K_{g}}$ into two spaces \(U=\overline{\Sg^3\smallsetminus V(K_{g-1})}\) and \(W=\overline{V(K_{g-1})\smallsetminus V(K_{g})}\) with the same properties as in the case \(g=2.\) By induction,  $\pi_1(U)$ admits a presentation like the one from the induction hypothesis and $$\pi_1(W)=\langle \ P_{g},\ Q_{g},L\ :\ Q_{g}^{n_{g}}P_{g}^{-\frac{a_{g}}{e_{g}}}=L^{n_{g}},\hspace{1mm} \left[L,P_g \right]=1\ \rangle.$$
	
	Then the meridian of $U\cap W$ is the class of the loop $P_g$ and as in the case \(g=2\) we have $P_g=Q_{g-1}^{x_{g-1}}Q_{g-2}^{-n_{g-2}x_{g-1}}P_{g-1}^{-y_{g-1}}$ in $\pi_1(U)$. As before, the longitude is the loop $L$ in  $\pi_1(W)$ and it corresponds to the class in $\pi_1(U)$ of a topological standard longitude of $\partial V(K_{g-1})$ that is $L=Q_{g-1}^{n_{g-1}}P_{g}^{-n_{g-1}\frac{a_{g-1}}{e_{g-1}}}$.
	\medskip
	
	Finally, similarly to the case of \(g=2\) we obtain the relation
	\begin{equation*}
		Q_g^{n_g}=P_g^{b_g}Q_{g-1}^{n_{g-1}n_g}
	\end{equation*}
	from which the result is proven.
\end{proof}

\begin{rem}
	The semigroup of values of an irreducible plane curve singularity is a free semigroup with the extra condition \(n_ia_i<a_{i+1}\) for all \(i\geq 1,\) see for example \cite{zariski}. In that case, Theorem \ref{teo:iterated_knot_fundamental_group} recovers Zariski's result about the fundamental group of the knot complement associated to the plane curve \cite{Zar32}. 
\end{rem}

At this point, it is natural to ask if this family contains more than one isotopy class of knots. The following example shows that there are free semigroups \(S\) for which the family \(\mathcal{K}_S\) contains different non-isotopic knots. 

\begin{ex}\label{ex:firstexample}
	Take the set $\{10,15,18\}$. One can check that the semigroup generated by this set is free for every arrangement, i.e. 
	\[
	\mathcal{A}=\left\{ (15, 10, 18), (15, 18, 10), (10, 15, 18), (10, 18, 15), (18, 15, 10), (18, 10, 15) \right\}
	\]
	
	In particular, we are going to focus on the arrangements  $\mathcal{G}_1:=\{18,15,10\}$, $\mathcal{G}_2:=\{10,18,15\}$ and $\mathcal{G}_3:=\{15,10,18\}$. In these cases, the iterated torus knots have cabling invariants \(\{(6,5),(3,10)\}\), \(\{(5,9),(2,15)\}\) and $\{(3,2),(5,18)\}$ for the arrangements $\mathcal{G}_1$,$\mathcal{G}_2$ and $\mathcal{G}_3$ respectively. 
	In fact, one can easily check using their splice diagram and the results in \cite[Chap. II]{EN} that 
	the rest of the knots in the family are isotopic to one of the previous knots. 
	\medskip
	
	We are going to show that, in fact, there are exactly three isotopic classes. To show that \(K_{\mathcal{G}_1},\) \(K_{\mathcal{G}_2}\) and \(K_{\mathcal{G}_2}\) are non-isotopic can be proved by looking at the (unique) torus decomposition of the knot exterior. Since the three knots have distinct torus knot spaces as one of their pieces they are non-isotopic.
	\medskip
	
	Alternatively, we can use the $\rho_{ab}-$invariant of a knot, we refer to \cite{rhoInvariant11}, where the author calls it $\rho_0$, to recall its definition and its formula in the case of iterated torus knots.  Following \cite[Corollary 2.10]{rhoInvariant11} we can easily compute the $\rho_{ab}-invariant$, see Table \ref{tab:rho_ab_example}, for the iterated torus knots mentioned above.
	\begin{table}[H]
		\centering
		\begin{tabular}{|c|c|c|c|}
			\hline
			&  $K_{\mathcal{G}_1}$ & $K_{\mathcal{G}_2}$ & $K_{\mathcal{G}_3}$\\ \hline
			& & & \\
			$\rho_{ab}$ & $\frac{-272}{15}$ & $\frac{-976}{45}$ & $\frac{-1352}{45}$ \\ 
			& & & \\\hline   \end{tabular}
		\caption{$\rho_{ab}-$invariant of the knots.}
		\label{tab:rho_ab_example}
	\end{table}
	Since the values of the \(\rho_{ab}\)--invariant differ for each knot, the three knots are non-isotopic to each other.
	\medskip

\end{ex}

As we have just seen, the family \(\mathcal{K}_S\) associated to a free numerical semigroup can contain different non-isotopic knots. However, we will see that the numerical semigroup \(S\) is topologically realizable in an unique way. The notion of topologically realizable numerical semigroup was defined by the first named author in \cite{Acisurvey} and it was suggested in the case of irreducible plane curves by him and Moyano-Fernández in \cite{AMknots}. Let us briefly recall its definition. 

A numerical semigroup has an associated semigroup algebra via the following graded homorphism: 

\[\begin{array}{ccc}
	\mathbb{C}[x_1,\dots,x_g]&\xrightarrow{\varphi} &\mathbb{C}[S] \\
	x_i&\mapsto &t^{a_i}
\end{array}\]
where the polynomial ring \(\mathbb{C}[x_1,\dots,x_g]\) is graded with \(\deg x_i=a_i.\) According to Herzog \cite{herzog}, the kernel \(I_S=\ker(\varphi)\) is a toric ideal and \(\mathbb{C}[S]\) is identified with the image of \(\varphi,\) i.e. \(\mathbb{C}[x_1,\dots,x_g]/I_S.\) 

On the other hand, given a knot \(K\), the abelianization of its knot group is always isomorphic to \(\mathbb{Z},\) and we can think in \(\mathbb{Z}\) as the multiplicative abelian group on the symbol \(\{t\}.\) Let \(G=\pi_1(X)\) be the fundamental group of the knot complement. Assume we represent \(\mathbb{Z}=<t>\) as a multiplicative group. We will say that a numerical semigroup \(S=\langle a_1,\dots, a_g\rangle\) is topologically realizable if there exists a knot \(K\) such that the following conditions holds:
\begin{enumerate}[leftmargin=1.5 em]
	
	\item There exists \(G=\pi_1(S^3\setminus K)=\langle \alpha_1,\dots,\alpha_s,\;| \ r_i(\alpha_1,\dots,\alpha_s)\;\text{for}\; i=1,\dots,l\rangle\) a presentation of the group \(G\) with \(s\) generators and \(l\) binomial relations in separated generators, i.e.
	\[r_i(\alpha_1,\dots,\alpha_s):\quad p_{i,1}(\alpha_{i_{1},\dots,i_{k}})=p_{i,2}(\alpha_{j_{1},\dots,j_{t}}) \quad \text{with}\ \{\alpha_{i_{1},\dots,i_{k}}\}\cap \{\alpha_{j_{1},\dots,j_{t}}\}=\emptyset, \ t+k\leq g+1 \ \] 
	and such that \(s=g\) being the number of minimal generators of \(I_S,\) the defining ideal of \(\mathbb{C}[S].\)
	\item The abelianization map is such that \(\alpha_i\mapsto t^{a_i}.\)
	\item If we consider the ring of polynomials \(\mathbb{C}[x_1,\dots,x_s]\) with the degrees induced by the abelianization map, then the relations of the group generate the defining ideal of \(S,\) i.e. 
	\[I_S=(p_{1,1}(x_1,\dots,x_s)-p_{1,2}(x_1,\dots,x_s),\dots,p_{l,1}(x_1,\dots,x_s)-p_{l,2}(x_1,\dots,x_s)).\]
\end{enumerate}

\begin{proposition}\label{prop:topologicallyrealizable}
	Any free numerical semigroup \(S\) is topologically realizable. Moreover,  the graded algebra associated to the fundamental group is independent of the arrangement of the set of generators for which \(S\) is free.
\end{proposition}
\begin{proof}
	To prove this result, it will be enough to show that the expression of the generators of the knot group in the abelianized knot group is independent of the chosen arrangement for the semigroup, that is, independent of the chosen knot which represent the semigroup.
	
	Consider \(S=\langle a_0,a_1,\dots,a_g\rangle \) a free semigroup for the arrangement $\{a_0,a_1,\dots,a_g\}$. Also consider the knot associated to this arrangement, i.e, the iterated torus knot having cabling invariants $\{(n_i,\frac{a_i}{e_i})\}_{i=1,\dots, g }.$ We know from Theorem \ref{teo:iterated_knot_fundamental_group} that the fundamental group of this knot admits a presentation with $g+1$ generators, $P_1,Q_1,\cdots,Q_{g}$, and $g$ relations:
	\begin{equation}
		Q_i^{n_i}=P_i^{b_i}Q_{i-1}^{n_{i-1}n_i} \hspace{3cm}(i=1,2,\dots,g\hspace{1mm};\hspace{1mm}Q_0=1)
	\end{equation}
	and the elements $P_2,\dots,P_{g}$ are defined by the relations:
	\begin{equation}
		P_{i+1}P_i^{y_i}Q_{i-1}^{n_{i-1}x_i}=Q_i^{x_i}, \hspace{3cm}(i=1,2,\dots,g-1) 
	\end{equation}
	where $x_i,y_i\in\mathbb{Z},\hspace{1mm} x_i>0,$ defined applying the Bezout identity to the pairs $(n_i,b_i)$, i.e, satisfying $x_ib_i=y_in_i+1.$
	\medskip
	
	It is a well known result that the abelianized fundamental group of a knot complement is isomorphic to $\mathbb{Z}$ and that we can take as a generator a meridian of a tubular neighborhood of the knot, namely $P_{g+1}$.  With the purpose to simplify the text we are going to denote this generator by $t$, so we have that all elements of the abelianized knot group are powers of $t$.
	Now we will use the relations from the presentation shown above of the knot group to calculate the expression as powers of $t$ of the generators. 
	
	First, remember that we obviously have the homology relation
	\begin{equation}
		P_i=P_{i+1}^{n_i} \text{ for } i=1,\dots,g.
	\end{equation}
	Then we have the following expression as powers of $t$
	\begin{equation}
		P_i=t^{n_in_{i+1}\cdots n_g}=t^{e_{i-1}} \text{ for } i=1,\dots,g,
	\end{equation}
	so we obtain $P_1=t^{e_0}=t^{a_0}$.
	
	To compute the expression of the generators $Q_1,\dots,Q_g$ we will use induction. Since $P_1^{\frac{a_1}{e_1}}=Q_1^{n_1}$ by Theorem \ref{teo:iterated_knot_fundamental_group}, we have that $Q_1=t^{a_1}$ in the abelianized fundamental group. So we assume as induction hypothesis that $Q_i=t^{a_i} \text{ for } i=1,\dots, k-1$ and we prove it for $k$. For this purpose take the relation:
	\begin{equation}\label{eq:hypotesis_abelianized}
		Q_k^{n_k}=P_k^{b_k}Q_{k-1}^{n_kn_{k-1}}
	\end{equation}
	then we have by induction hypothesis $Q_{k-1}=t^{a_{k-1}}$, and $P_k=t^{e_{k-1}}$. Also remember that we defined $b_k=\frac{a_k}{e_k}-n_kn_{k-1}\frac{a_{k-1}}{e_{k-1}}$, so from the relation (\ref{eq:hypotesis_abelianized}) we obtain:
	\begin{equation}
		Q_k^{n_k}=t^{n_{k}a_k-n_kn_{k-1}a_{k-1}}t^{n_kn_{k-1}a_{k-1}}=t^{n_ka_k}
	\end{equation}
	and the result is proven. 
\end{proof}

\section{Some properties of the family of knots associated to a free semigroup}

The knot group is a very powerful invariant for distinguishing two knots, the only problem is that determining whether two groups are isomorphic is not usually an easy task. Because of this issue, frequently we use other invariants that derive from the knot group. The most common of these invariants is the Alexander polynomial which we will denote by $\Delta_K(t)$.

In this section, we will provide some properties of the family of knots associated to a fixed free semigroup. In particular, we will show that all knots in the family have the same Alexander polynomial and genus.

\subsection{Alexander polynomial and Poincaré series} \label{sec:alexander}
The definition of the Alexander polynomial is based on the notion of universal abelian covering \(\rho:\widetilde{X}\rightarrow X\) of the knot complement \(X=S^3\setminus K.\) The group of covering transformations \(H_1(X;\mathbb{Z})=\mathbb{Z}\) is a free abelian multiplicative group on the symbol \(\{t\}\) where \(t\) is geometrically associated with an oriented meridian of the knot. In this way, if \(\widetilde{p}\) is a typical fiber of \(\rho\) then the group \(H_1(\widetilde{X},\widetilde{p};\mathbb{Z})\) becomes a module over \(\mathbb{Z}[t^{\pm 1}].\) The Alexander polynomial \(\Delta_K(t)\) is then defined as the greatest common divisor of the first Fitting ideal \(F_1(H_1(\widetilde{X};\mathbb{Z})).\) Let \(S=\langle a_0,\dots,a_g\rangle\) be a free semigroup for the arrangement \(\{a_0,\dots,a_g\}.\) In the previous section we showed that there exists an iterated torus knot having cabling invariants $\{(p_i,q_i)\}_{i=1,\dots, g }=\{(n_i,\frac{a_i}{e_i})\}_{i=1,\dots, g }.$ Following \cite[Chap IV. 17 (a)]{EN}, we can compute the Alexander polynomial from the splice diagram if we enlarge the labeling in the diagram as follows:

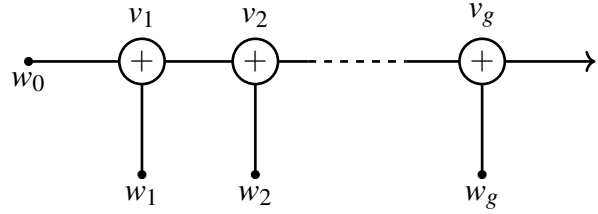
\begin{figure}[H]
	\centering
	\begin{tikzpicture}
		
		\draw[line width=1pt] (0,0) circle [radius=0.3cm] node[above,yshift=0.3cm] {\(v_1\)};
		\draw[line width=1pt] (1.5,0) circle [radius=0.3cm] node[above,yshift=0.3cm] {\(v_2\)};
		\draw[line width=1pt] (4.5,0) circle [radius=0.3cm] node[above,yshift=0.3cm] {\(v_g\)};
		
		\draw[line width=1pt] (-1.5,0) -- (-0.3,0);
		\draw[line width=1pt] (0.3,0) -- (1.2,0);
		\draw[line width=1pt] (0,-1.5) -- (0,-0.3);
		\draw[line width=1pt] (1.8,0) -- (2.2,0);
		\draw[line width=1pt,dashed] (2.2,0) -- (3.5,0);
		\draw[line width=1pt] (3.5,0) -- (4.2,0);
		\draw[line width=1pt,->] (4.8,0) -- (6,0);
		\draw[line width=1pt] (1.5,-1.5) -- (1.5,-0.3);
		\draw[line width=1pt] (4.5,-1.5) -- (4.5,-0.3);
		
		\filldraw (-1.5,0) circle [radius=1.5pt] node[below] {\(w_0\)};
		\filldraw (0,-1.5) circle [radius=1.5pt] node[below] {\(w_1\)};
		\filldraw (1.5,-1.5) circle [radius=1.5pt] node[below] {\(w_2\)};
		\filldraw (4.5,-1.5) circle [radius=1.5pt] node[below] {\(w_g\)};
		
		\node at (0,0) {$+$};
		\node at (1.5,0) {$+$};
		\node at (4.5,0) {$+$};
	\end{tikzpicture}
	\caption{Splice diagram of iterated torus knot}
	\label{fig:spliceenlarged}
\end{figure}
The numbers \(w_i,v_i\) have the following expression in terms of the cabling invariants: 
\[w_0=p_1\cdots p_n,\; w_i=q_ip_{i+1}\cdots p_n \quad \text{and}\quad v_i=q_ip_i\cdots p_n.\]

In this way, from the definition of the cabling invariants of a knot associated to \(S,\) we have that \(w_i=a_i\) and \(v_i=n_ia_i\) for all \(i.\) This leads to the following property.
\begin{proposition}\label{prop:linkinvariants}
	Let \(\mathcal{K}_S\) be the family of knots associated to a free semigroup \(S=\langle a_0,\dots,a_g\rangle.\) If \(K\in \mathcal{K}_S\) then \(w_i=a_i\) for all \(i=0,\dots,g\) and \(v_i=n_ia_i\) for all \(i=1,\dots,g.\)
	
	In particular, the numbers \(w_i,v_i\) are equal up to reordering for all members in the family \(\mathcal{K}_S.\)
\end{proposition}

Remember that for a numerical semigroup $S$, $P_S(t)$ denotes its generating series. Then, as a consequence of Proposition \ref{prop:linkinvariants}, we can provide the following theorem. 
\begin{theorem}\label{thm:AlexanderPoincare}
	Let \(S=\langle a_0,\dots,a_g\rangle\) be a free numerical semigroup and let \(\mathcal{K}_S\) be the family of associated knots. Then for all \(K\in\mathcal{K}_S\) we have
	\[\Delta_{K}(t)=P_S(t)(1-t).\]
	In particular, all knots associated to a free numerical semigroup have the same Alexander polynomial.
\end{theorem}
\begin{proof}
	Let \(K_i\in \mathcal{K}_S.\) According to \cite[Chap IV]{EN} we have 
	\[\Delta_{K_i}(t)=(1-t)\frac{\prod_{i=1}^g(1-t^{v_i})}{\prod_{i=0}^g(1-t^{w_i})}.\]
	By Proposition \ref{prop:linkinvariants} we have \(v_i=n_ia_i\) and \(w_i=a_i\) for all \(i.\) Therefore, \(\Delta_{K_i}(t)=\Delta_{K_j}(t)\) for all \(K_i,K_j\in\mathcal{K}_S.\)
	On the other hand, it is well known (see for example \cite{Acisurvey} and the references therein) 
	\[P_S(t)=\frac{\prod_{i=1}^g(1-t^{n_ia_i})}{\prod_{i=0}^g(1-t^{a_i})},\]
	from which we conclude the proof.
\end{proof}
Recall that iterated torus knots are fibered. This implies that the degree of the Alexander polynomial is twice the Seifert genus of the knot, \(\deg(\Delta(t))=2g(K)\) (see for example \cite[Prop. 8.16]{BurdeKnots}). On the other hand, recall that a numerical semigroup \(S\) has a conductor, i.e. \(c\in \mathbb{N}\) such that \(c+\mathbb{N}\subset S\) and \(c-1\notin S.\) Recall also, that in the case of free numerical semigroups the graded algebra associated to \(R=\mathbb{C}[S]\) is a Gorenstein ring and thus \(c=2\delta(S),\) where \(\delta(S)=\dim_\mathbb{C}\overline{R}/R\) and \(\overline{R}\) being the normalization of \(R;\) we refer to the survey \cite{Acisurvey} and the references therein for further details. An important corollary is 
\begin{corollary}\label{cor:deltagenus}
	Let \(\mathcal{K}_S\) be the family of knots associated to a free numerical semigroup \(S.\) Then for any \(K\in\mathcal{S}\) we have 
	\[g(K)=\delta(S).\]
\end{corollary}
\subsection{Some relevant families}
To finish the paper, we will review some relevant families of iterated torus knots in order to show the richness of the family of knots associated to a free semigroup.
\subsubsection{Algebraic knots}\label{sec:algebraicknots}
Let \((C,0)\subset(\mathbb{C}^2,0)\) be a germ of an isolated irreducible plane curve singularity. Let \(R\) denote its local ring at the origin and let \(\overline{R}\simeq\mathbb{C}\{t\}\) be its normalization. The normalization \(R\hookrightarrow\overline{R}\) induces a discrete valuation \(v:R\rightarrow\mathbb{Z}\) from which \(S=v(R)\) has a natural structure of numerical semigroup (see \cite{zariski}). The minimal generators of \(S_C=\langle a_1,\dots,a_g\rangle\) are computed from a parameterization of the curve \cite{zariski} and satisfy the following condition:
\begin{equation*}
	\label{cond:beta} S_C\; \text{is free and}\; n_{i}a_i<a_{i+1}\; \text{for}\; 2\leq i\leq g-1.
\end{equation*}

On the other hand, associated to \(C\) there is the knot \(K_C=C\cap\mathbb{S}^3.\) The associated knot \(K_C\) is an iterated torus knot whose cable invariants can be computed in terms of the minimal generators, which in fact are determined by the Puiseux parametrization of the curve; we refer to \cite{EN,Zar32,zariski} for further details. The iterated torus knots arising from an isolated irreducible plane curve singularity are called algebraic knots. 

As the free condition is less restrictive, one can easily check that the set of all possible arrangements for which \(S_C\) is free can be quite large. Therefore, a natural question is to ask whether any \(K\in \mathcal{K}_S\) is isotopically equivalent to \(K_C.\) In this direction, the following example shows that one can easily obtain examples for which the answer to this question is negative. 

Let us consider \(C:f(x,y)=(y^5-x^7)^4+x^{29}.\) Its semigroup of values is \(S_C=\langle 20, 28, 145\rangle.\) Following Delorme \cite{delormeglue}, it is easy to check that the arrangements for which  \(S_C\) is free are 
\[\{20, 28, 145\},\{28, 20, 145\},\{20, 145, 28\},\{145, 20, 28\}.\]
As in Example \ref{ex:firstexample}, we only need to examine \(\mathcal{G}_1=\{20, 28, 145\}\) and \(\mathcal{G}_2=\{20, 145, 28\}.\) According to Zariski \cite{Zar32}, the knot \(K_C\) coincides with the knot provided by our Theorem \ref{teo:iterated_knot_fundamental_group} for the arrangement \(\mathcal{G}_1.\) On the other hand, our Theorem \ref{teo:iterated_knot_fundamental_group} also produces the knot \(K_{\mathcal{G}_2}\) associated to \(\mathcal{G}_2.\) As in Example \ref{ex:firstexample}, if we look at the \(\rho_{ab}\)--invariants, we have \(\rho_{ab}(K_C)\neq \rho_{ab}(K_{\mathcal{G}_2}).\) Therefore, the knots are non--isotopic.

Also, according to the classification of algebraic knots provided in \cite{EN}, \(K_{\mathcal{G}_2}\) is not an algebraic knot. This shows that the isotopic class of the knot \(K_C\) associated to the curve is not unique in the set \(\mathcal{K}_{S_C}\) of knots associated to the free semigroup \(S_C.\) Even more, this class contains knots which are not algebraic.

\subsubsection{L-space knots}\label{sec:Lspace}
The notion of \(L\)--space was first defined in \cite{OStopology}. An \(L\)--space is a rational homology sphere having the Heegaard Floer homology of a lens space. A knot \(K\) in \(\mathbb{S}^3\) is an \(L\)--space knot if it admits an \(L\)--space surgery. It is not our purpose to go into the details of this definition, therefore we refer the reader to \cite{OSadvances,OStopology} for further details.

From the point of view of cable knots, Hom  \cite{Hom11} proved the following property for \(L\)--space knots.
\begin{theorem}\cite{Hom11}\label{teo:HOM11}
	The \((p,q)\)--cable of a knot \(K\subset \mathbb{S}^3\) admits a positive L-space surgery if and only if K admits a positive L-space surgery and \(q/p\geq 2g(K)-1\), where \(g(K)\) is the Seifert genus of \(K\).
\end{theorem}
On the other hand, any torus knot is an \(L\)--space knot \cite{OStopology}. Therefore, Hom's result provides a useful way to check which knots of our family are \(L\)--space knots. As we already mentioned, a free numerical semigroup \(S=\langle a_0,\dots, a_g\rangle\) can be constructed recursively from the pairs \((n_i,a_i/e_i).\) As in the beginning of Section \ref{sec:topologyknotfree}, let us denote \(S_i\) the truncated semigroups of the arrangement. Obviously, each \(S_i\) is a free numerical semigroup. Also, we have seen in Corollary \ref{cor:deltagenus} that twice the Seifert genus of the knot equals the conductor of the semigroup \(2g(K)=c(S).\) As \(S_1\) is always a torus knot, we have the following immediate characterization of which knots in \(\mathcal{K}_S\) are \(L\)--space knots. 
\begin{proposition}\label{prop:charL-space}
	Let \(S=\langle a_0,\dots, a_g\rangle\) be a free numerical semigroup and \(S_i=\langle a_0,\dots, a_i\rangle/e_i\) be its truncated semigroups. Let \(K\) be the knot associated to the arrangement \(\{a_0,\dots, a_g\}.\) Then, \(K\) is an \(L\)--space knot if and only if \(a_{i+1}/e_i\geq c(S_i)-1\) for all \(i=1,\dots,g-1.\)
\end{proposition}
\begin{proof}
	Observe that according to Theorem \ref{teo:iterated_knot_fundamental_group} \(K\) is constructed as follows: we consider \(K_1\) the \((n_1,a_1/e_1)\)--torus knot. Then for \(i\geq 2\) we denote \(K_i\) the knot obtained by a \((n_i,a_i/e_i)\)--cable on \(K_{i-1}.\)  By Theorem \ref{teo:HOM11}, \(K_{i}\) is an \(L\)--space knot if and only if \(K_{i-1}\) is an \(L\)--space knot and \(a_i/e_{i-1}\geq c(S_{i-1})-1.\) The results follows from the fact that \(K_1\) is an \(L\)--space knot.
\end{proof}

Proposition \ref{prop:charL-space} provides a powerful tool to look for examples of iterated torus knots which are not \(L\)--space knots. First, we recall that, by Delorme \cite[Prop. 10]{delormeglue}, the conductor of a free numerical semigroup \(S=\langle a_0,\dots,a_g\rangle\) has the following expression
\[c(S)=\sum_{i=1}^g(n_i-1)a_i-a_0+1.\]
From this expression, a first well known consequence is that any algebraic knot is an \(L\)--space knot. This is due to the condition \(n_ia_i<a_{i+1}\) as in that case we have 

\[\frac{a_{i+1}}{e_i}-\sum_{j=1}^i(n_j-1)\frac{a_{j}}{e_i}+\frac{a_{0}}{e_i}=\frac{a_{i+1}}{e_i}-\frac{n_ia_{i}}{e_i}+\frac{a_{i}}{e_i}-\frac{n_{i-1}a_{i-1}}{e_i}\pm \cdots+\frac{a_1}{e_i}+\frac{a_{0}}{e_i}\geq 0.\]
On the other hand, let us come back to the semigroup \(S_C=\langle 20, 28, 145\rangle\) of the plane curve \((y^5-x^7)^4+x^{29}.\) As we mentioned above, in the set of associated knots \(\mathcal{K}_{S_C}\) besides the isotopy class of the knot \(K_C\) we have the isotopy class of the knot \(K_{\mathcal{G}_2}\) with cabling invariants \(\{(4,29),(5,28)\}.\) In this case we have  \(c(S_1)-1=3\cdot 28-1=83>28/5.\) Therefore \(K_{\mathcal{G}_2}\) is not an \(L\)--space knot. This provides an alternative argument to show that \(K_C\) and \(K_{\mathcal{G}_2}\) are not isotopic knots.

\section*{Declarations}

\textbf{Conflict of interest}: The authors declare no conflict of interest.

\textbf{Ethics approval}: Not applicable.

\textbf{Consent to participate}: Not applicable.

\textbf{Consent for publication}: Not applicable.

\textbf{Data availability}: All data generated or analyzed during this study are included in this published article.

\textbf{Code availability}: Not applicable

\bibliographystyle{plainurl}
\bibliography{biblioknots}

\end{document}